\documentclass[12pt]{article}
\usepackage{fullpage}
\usepackage[english]{babel}
\usepackage{amssymb}
\usepackage{graphicx}
\usepackage{epsfig}

\usepackage{amsmath,amsthm,verbatim}
\usepackage{amsrefs}
\usepackage{hyperref}
\usepackage{setspace}

\newtheorem{theorem}{Theorem}

\newtheorem{lemma}[theorem]{Lemma}
\newtheorem{corol}[theorem]{Corollary}

\theoremstyle{definition}
\newtheorem{defi}[theorem]{Definition}

\theoremstyle{remark}

\title{On obtaining simple identities for  overshoots of spectrally negative L\'evy processes}
\author{R. L. Loeffen\footnote{School of Mathematics, University of Manchester, Oxford Road, Manchester M13 9PL, United Kingdom, e-mail: ronnie.loeffen@manchester.ac.uk}}

\begin{document}
\maketitle
\begin{abstract}
For a (killed) spectrally negative L\'evy process  we provide an analytic expression for the distribution of its overshoot over a fixed level in terms of the infinitesimal generator and the scale function of the process.   Our identity involves an auxiliary function and the simplicity of the identity depends very much on the choice of this function. In particular, for specific choices  one recovers various previous established formulas in the literature. We review several applications and also show that one can get in  a similar way identities of overshoots for reflected and refracted spectrally negative L\'evy processes.
\end{abstract}

\begin{quote}
\begin{small}
\end{small}
\end{quote}

\section{Introduction}
Let $X=\{X_t:t\geq0\}$ be a spectrally negative L\'evy process on the filtered probability space $(\Omega, \mathcal F,\{\mathcal F_t:t\geq0\},\mathbb P)$, i.e. $X$ is a  process with stationary and independent increments and no positive jumps. Here we assume that $\{\mathcal F_t:t\geq0\}$ is the $\mathbb P$-completion of the natural filtration of $X$, so that it satisfies the usual conditions, see e.g. \cite{protter_2nded}*{Theorem I.31}. We further exclude the case that $X$ is the negative of a subordinator, i.e. we exclude the case of $X$ having decreasing paths. The law of $X$ such that $X_0 = x$ is denoted by $\mathbb P_x$ and the corresponding expectation by $\mathbb E_x$.
We work with the first passage times
 \begin{equation*}
\tau_a^+ = \inf\{t>0:X_t>a\}, \quad \text{and} \quad \tau_a^- = \inf\{t>0:X_t<a\}.
\end{equation*}
In this paper we are interested in the following expectation concerning  the overshoot  over a level $a$, namely
\begin{equation}\label{gerbershiu}
 \mathbb E_x \Big[ \mathrm e^{-q \tau_a^-}  f(X_{\tau_a^-}) \mathbf 1_{\{ \tau_a^-<\tau_b^+ \}}  \Big],
\end{equation}
where $-\infty<a<b<\infty$, $q\geq0$, $x\in[a,b]$ and $f:(-\infty,a)\to\mathbb R$ is a (penalty) function satisfying certain regularity conditions which we specify later. When $b\to\infty$, \eqref{gerbershiu} is a specific case (namely, in which one does not consider the undershoot) of the so-called expected discounted penalty function  introduced by Gerber and Shiu \cite{gerbershiu}, which can be interpreted as a  risk measure for an insurance company. Further,  \eqref{gerbershiu}     appears in many applications of spectrally negative L\'evy processes in which circumstances change when a given level is crossed downwards. Some examples are solving exit problems for  refracted L\'evy processes, see e.g. \cite{kyploeffen}, determining Laplace transforms of occupation times, see e.g. \citelist{\cite{landriault_occup}\cite{loeffenrenaudzhou}\cite{kyp_occup}\cite{renaud_red}}, or computing the value of multi-band strategies, see  \cite{palmgerbershiu}. Analytic expressions for \eqref{gerbershiu} in terms of the scale function and the L\'evy triplet of the spectrally negative L\'evy process are known, but for specific penalty functions   it has been possible, usually after quite some effort, to come up with much simpler expressions; we will review some  examples later on in Section \ref{sec_examples}. 
In this paper we shed some light on why this is possible and show how one can directly get such simple expressions. The novelty of our approach is that we introduce  an extension $\widetilde f$ to the interval $(-\infty,b]$ of the function $f$, which is   free to choose (as long as it satisfies certain regularity conditions) and then provide an expression for \eqref{gerbershiu} which depends on this extension.
  The existing expressions found for \eqref{gerbershiu} turn out to correspond to particular choices of $\widetilde f$. We will further see that  given a penalty function, it is quite obvious how to choose the extension  in such a way that the resulting formula is as simple as possible.

 The rest of the paper is organised as follows. In the  next section we briefly review some background on spectrally negative L\'evy processes and state the main result whose proof is given in Section \ref{sec_proof}. The main idea of the proof also applies to  certain modifications of  spectrally negative L\'evy processes and in Section \ref{sec_reflecrefrac} we give the analogues of \eqref{gerbershiu} for reflected and refracted spectrally negative L\'evy processes. Finally, in Section \ref{sec_examples} we go over some examples to illustrate the usefulness of our result.

\section{Main result}
Before we state the main result, we briefly give some background information on spectrally negative L\'evy processes and their scale functions; proofs can be found in for instance the book of Kyprianou \cite{kypbook}. A spectrally negative L\'evy process is characterised in terms of a so-called  L\'evy triplet $(\gamma,\sigma,\Pi)$,  where $\gamma \in \mathbb R$, $\sigma \geq 0$ is called the Gaussian coefficient and  $\Pi$, called the L\'evy measure, is a  measure on $(0,\infty)$ satisfying
\begin{equation}\label{levymeasure}
\int^{\infty}_0 (1 \wedge \theta^2) \Pi(\mathrm{d}\theta) < \infty.
\end{equation}
Note that for convenience we define the L\'evy measure in such a way that it is a measure on the positive half line instead of the negative half line.
As the L\'{e}vy process $X$ has no positive jumps, its Laplace transform exists and is given by
\begin{equation*}
\mathbb E \left[ \mathrm{e}^{\lambda X_t} \right] = \mathrm{e}^{t \psi(\lambda)} ,
\end{equation*}
for $\lambda,t \geq 0$, where
\begin{equation*}
\psi(\lambda) = \gamma \lambda + \frac{1}{2} \sigma^2 \lambda^2 + \int^{\infty}_0 \left( \mathrm{e}^{-\lambda \theta} - 1 + \lambda \theta \mathbf 1_{\{\theta\leq 1\}} \right) \Pi(\mathrm{d}\theta).
\end{equation*}
The process $X$ has paths of bounded variation if and only if $\sigma=0$ and $\int^{1}_0 \theta \Pi(\mathrm{d}\theta)<\infty$. 

We now recall the definition of the $q$-scale function $W^{(q)}$. For $q \geq 0$, the $q$-scale function of the process $X$ is defined on $[0,\infty)$ as the continuous function with Laplace transform on $[0,\infty)$ given by
\begin{equation}\label{def_scale}
\int_0^{\infty} \mathrm{e}^{- \lambda y} W^{(q)} (y) \mathrm{d}y = \frac{1}{\psi(\lambda) - q} , \quad \text{for $\lambda$ sufficiently large.} 
\end{equation}
This function is unique, positive and strictly increasing for $x\geq0$. We extend $W^{(q)}$ to the whole real line by setting $W^{(q)}(x)=0$ for $x<0$. 
Scale functions appear in various fluctuation identities and  we will need the following one involving exiting the interval $[a,b]$ at the upper boundary,
\begin{equation}\label{twosidedexit}
\mathbb E_x \left[ \mathrm{e}^{-q \tau_b^+ } \mathbf 1_{\{\tau_b^+<\tau_a^-\}} \right] = \frac{W^{(q)}(x-a)}{W^{(q)}(b-a)}, \quad \text{$x\leq b$},
\end{equation}
as well as  the $q$-resolvent measure of $X$ killed upon exiting the interval $[a,b]$,
\begin{equation}\label{resolvent}
\int_{0}^{\infty}\mathrm{e}^{-qs}  \mathbb P_x(X_s\in\mathrm dz,s<\tau_a^-\wedge\tau_b^+)  \mathrm{d}s
= \left[ \frac{W^{(q)}(x-a)}{W^{(q)}(b-a)} W^{(q)}(b- z)  -  W^{(q)}(x-z) \right] \mathrm dz, 
\end{equation}
where $z\in[a,b]$ and $x\leq b$. We refer to Theorem 8.1 and   Theorem 8.7 of \cite{kypbook} for the proof and origin of these two identities. When the Gaussian coefficient $\sigma=0$, the process $X$ cannot creep downwards, i.e. $\mathbb P_x(\tau_a-=a)=0$ for all $x>a$. When $\sigma>0$, $X$ does creep downwards and we have
\begin{equation}\label{creeping}
\mathbb E_x \left[ \mathrm e^{-q \tau_a^-}   \mathbf 1_{\{X_{\tau_a^-}=a, \tau_a^-<\tau_b^+ \}}  \right]
= \frac{\sigma^2}2 \left( W^{(q)\prime}(x-a) - \frac{W^{(q)}(x-a)}{W^{(q)}(b-a)} W^{(q)\prime}(b-a)  \right), \quad \text{$a<x\leq b$}.
\end{equation}
Note that \eqref{creeping} can be deduced from \eqref{resolvent}, see the proof of Corollary 2 in   \cite{pistorius_potential}.

In order to state the main result, we need to introduce the following space of functions.
\begin{defi}\label{def_funcspace}
For $X$ a spectrally negative L\'evy process with L\'evy triplet $(\gamma,\sigma,\Pi)$ and $-\infty<a< b<\infty$, we define   $\mathcal H(X;a,b)$ as the function space consisting of   measurable, locally bounded functions  $h:(-\infty,b]\to\mathbb R$ such that   the following hold:
\begin{itemize}
\item[(i)] $h$ is continuous on $(a,b]$,
\item[(ii)] there exists $\lambda>0$ such that $x\mapsto \int_\lambda^\infty  h(x-\theta)  \Pi(\mathrm d\theta)$ is bounded on $(a,b)$,
\item[(iii)] if $X$ has paths of unbounded variation, then $h$ is continuously differentiable on $(a,b)$ with the derivative being absolutely continuous on $(a,b)$ and having a density that is bounded on  $(a,b)$. 

\item[(iv)] if $X$ has paths of bounded variation, then $h$ is absolutely continuous on $(a,b)$ with a density that is bounded and of bounded variation  on  $(a,b)$.
\end{itemize}
\end{defi}
For   $h\in\mathcal H(X;a,b)$ we define  $\mathcal A h(x)$ by
\begin{equation*}
 \begin{split}
   \mathcal A h(x) = & \gamma h'_-(x) + \frac12 \sigma^2 h''(x) + \int_0^\infty [h(x- \theta)-h(x)+h'_-(x) \theta \mathbf 1_{\{ \theta\leq 1\}} ] \Pi(\mathrm d \theta), \quad x\in(a,b),
 \end{split}
\end{equation*}
where $h'_-$ denotes the left-derivative of $h$ and where if $\sigma=0$, the term $\frac12 \sigma^2 h''(x)$ is understood to equal zero and if $\sigma>0$, $h''$ denotes a version of the density of $h'$ (for all the results derived in this paper, it does not matter which version is chosen).  Note that if  $h\in\mathcal H(X;a,b)$ and $X$ has paths of bounded variation, then $h$ on $(a,b)$ can be written as the difference of two convex functions. This means that when   $h\in\mathcal H(X;a,b)$, the right- and left-derivative $h'_+(x)$ and $h'_-(x)$ are always well-defined for  $x\in(a,b)$. Further, note that the integral term in  $\mathcal A h(x)$ is also well-defined when  $h\in\mathcal H(X;a,b)$, see e.g. the proof of Lemma \ref{lemma_main} below. It is well-known that the operator $\mathcal A$ coincides with the infinitesimal generator of $X$ on the space of twice continuously differentiable functions that vanish at infinity, together with the first two derivatives, cf. \cite{sato}*{Theorem 31.5}.

 We can now state the main theorem.
\begin{theorem}\label{thm_main}
Let $-\infty<a< b<\infty$, $q\geq0$  and $f:(-\infty,a]\to\mathbb R$ be a measurable, locally bounded   function  with the property that there exists $\lambda>b-a$ such that $x\mapsto \int_\lambda^\infty  f(x-\theta)  \Pi(\mathrm d\theta)$ is bounded on $(a,b)$. Let $\widetilde f:(-\infty,b]\to\mathbb R$ be an extension of $f$ that lies in $\mathcal H(X;a,b)$. Then for $a<x\leq b$,
\begin{equation}\label{mainidentity_general}
 \begin{split}
   \mathbb E_x  \Big[ \mathrm e^{-q \tau_a^-}  f(X_{\tau_a^-}) & \mathbf 1_{\{ \tau_a^-<\tau_b^+ \}}  \Big]  
 \\
 = & \widetilde f(x)  - \frac{W^{(q)}(x-a)}{W^{(q)}(b-a)}    \widetilde f(b) \\
& + \int_a^b   (\mathcal A-q) \widetilde f(z) \left[ \frac{W^{(q)}(x-a)}{W^{(q)}(b-a)} W^{(q)}(b- z)  -  W^{(q)}(x-z) \right] \mathrm dz \\
 & +  \left( f(a) - \widetilde f(a+) \right) \frac{\sigma^2}2 \left( W^{(q)\prime}(x-a) - \frac{W^{(q)}(x-a)}{W^{(q)}(b-a)} W^{(q)\prime}(b-a)  \right),
 \end{split}
\end{equation}
where $\widetilde f(a+):=\lim_{y\downarrow a}\widetilde f(y)$. Further, for the case $X_0=a$, if  $X$ has paths of bounded variation, then
\begin{equation*}
  \mathbb E_a \Big[ \mathrm e^{-q \tau_a^-}  f(X_{\tau_a^-}) \mathbf 1_{\{ \tau_a^-<\tau_b^+ \}}  \Big]  
=   \widetilde f(a+)  -       \frac{W^{(q)}(0)}{W^{(q)}(b-a)} \left[ \widetilde f(b)  - \int_a^b   (\mathcal A-q) \widetilde f(z)   W^{(q)}(b- z)  \mathrm dz \right],
\end{equation*}
whereas if $X$ has paths of unbounded variation, $\mathbb E_a \left[ \mathrm e^{-q \tau_a^-}  f(X_{\tau_a^-}) \mathbf 1_{\{ \tau_a^-<\tau_b^+ \}}  \right]=f(a)$.
\end{theorem}

\begin{corol}\label{corol_main}
Suppose   the conditions in Theorem \ref{thm_main} hold and assume in addition that one of the following holds:
\begin{itemize}
\item[(i)] $\int_0^1 \theta\Pi(\mathrm d\theta)<\infty$ and $\sigma=0$,
\item[(ii)] $\int_0^1 \theta\Pi(\mathrm d\theta)<\infty$ and $\widetilde f$ is right-continuous at $a$,
\item[(iii)] $\int_0^1 \theta\Pi(\mathrm d\theta)=\infty$  and $\widetilde f$ has a bounded density in a neighbourhood of $a$.
\end{itemize}  
Then for $a<x\leq b$,
\begin{equation}\label{mainidentity_simple}
 \begin{split}
   \mathbb E_x \Big[ \mathrm e^{-q \tau_a^-}  f(X_{\tau_a^-}) \mathbf 1_{\{ \tau_a^-<\tau_b^+ \}}  \Big]  
= & \widetilde f(x)  -   \int_a^x   (\mathcal A-q) \widetilde f(z)   W^{(q)}(x- z) \mathrm dz \\
& -  \frac{W^{(q)}(x-a)}{W^{(q)}(b-a)} \left[ \widetilde f(b)  - \int_a^b   (\mathcal A-q) \widetilde f(z)   W^{(q)}(b- z)  \mathrm dz \right].
 \end{split}
\end{equation}
\end{corol}

The identities \eqref{mainidentity_general} and \eqref{mainidentity_simple} depend on the chosen extension $\widetilde f$. Note that one always has the option to choose $\widetilde f$ to be right-continuous at $a$, with the result that the creeping term in \eqref{mainidentity_general} vanishes. Though  \eqref{mainidentity_simple} is more convenient to work with, it is not always possible to use it: in particular when  $\int_0^1 \theta\Pi(\mathrm d\theta)=\infty$ and $f$ is not left-continuous at $a$, then the extra condition in Corollary \ref{corol_main} does not hold no matter which extension is chosen and so one then has to resort to \eqref{mainidentity_general}.
When we choose $\widetilde f(y)=0$ for all $y\in(a,b]$, then \eqref{mainidentity_general}  becomes,
\begin{equation}\label{extis0}
 \begin{split}
   \mathbb E_x  \Big[ \mathrm e^{-q \tau_a^-}  f(X_{\tau_a^-}) & \mathbf 1_{\{ \tau_a^-<\tau_b^+ \}}  \Big]  
 \\
 = &   \int_a^b   \int_{z-a}^\infty   f(z-\theta) \Pi(\mathrm d\theta) \left[ \frac{W^{(q)}(x-a)}{W^{(q)}(b-a)} W^{(q)}(b- z)  -  W^{(q)}(x-z) \right] \mathrm dz \\
 & +   f(a)  \frac{\sigma^2}2 \left( W^{(q)\prime}(x-a) - \frac{W^{(q)}(x-a)}{W^{(q)}(b-a)} W^{(q)\prime}(b-a)  \right).
 \end{split}
\end{equation}
This identity is often  used in the literature to deal with overshoots, see e.g.  Equation (10.28) in \cite{kyploeffen}, (5) in \cite{loeffenrenaudzhou}, (2.6) in \cite{kyp_occup} or  in the case where $b\to\infty$, (4) in \cite{landriault_occup}. It can be proved by using the compensation formula and has the advantage that one can easily incorporate the undershoot $\lim_{s\uparrow\tau_a^-}X_{t}-a$ as well,  see e.g. \cite{kypbook}*{Equation (8.32)}.  However, when one is only interested in the overshoot, we recommend to use \eqref{mainidentity_general} or \eqref{mainidentity_simple} instead, since by choosing the extension $\widetilde f$ wisely, one can, in specific cases, get significantly simpler identities, which are not obvious to spot by using \eqref{extis0}, see Section \ref{sec_examples} for examples. Avram et al. \cite{palmgerbershiu} work with a different identity than \eqref{extis0}, see Definition 5.2 and Proposition 5.5 in \cite{palmgerbershiu}. Their identity corresponds to the extension $\widetilde f(y)=f'_-(a)y + f(a)$, $y\in(a,b]$, assuming that the penalty function $f$ admits a left-derivative at $a$. As  can be seen in \cite{palmgerbershiu}, this choice is very convenient for computing the value function of multi-band strategies as in that case one deals with penalty functions which are affine in a neighbourhood of $a$.

\section{Proof}\label{sec_proof}

\begin{lemma}\label{lemma_main}
Let $-\infty<a< b<\infty$ and $h\in\mathcal H(X;a,b)$. Assume further that (i) $\int_0^1 \theta\Pi(\mathrm d\theta)=\infty$  and $h$ has a bounded density in a neighbourhood of $a$  or  (ii) $\int_0^1 \theta\Pi(\mathrm d\theta)<\infty$. Then $\int_a^b | \mathcal A h(x) |\mathrm dx<\infty$.  
\end{lemma}
\begin{proof}[\textbf{Proof}] 
We first consider the case where condition (i) holds. By Definition \ref{def_funcspace}(iii), $h'$ has a bounded density on $(a,b)$ and we denote by $h''$ a version of this density.  
Then by the boundedness of $h''$ on $(a,b)$, $h$ and $h'$ are  bounded on $(a,b)$. Hence   $\int_a^b |\gamma h'(x) + \frac12 \sigma^2 h''(x) |\mathrm dx<\infty$.  To deal with the integral term, first note that by Taylor's theorem   we have for all $x\in(a,b)$,
\begin{equation*}
\begin{split}
  \int_0^{x-a} |h(x- \theta)-h(x)+h'(x)\mathbf 1_{\{ \theta\leq 1\}} | \Pi(\mathrm d \theta)   \leq & \sup_{t\in(a,x)} |h''(t)| \int_0^{x-a} \frac12  \theta^2\Pi(\mathrm d \theta) \\
  \leq & \sup_{t\in(a,b)} |h''(t)| \int_0^{b-a} \frac12  \theta^2\Pi(\mathrm d \theta).
\end{split}
\end{equation*}
Second, by Fubini,
\begin{equation*}
 \begin{split}
\int_a^b     \int_{x-a}^\infty |h'(x)  \theta\mathbf 1_{\{\theta\leq 1\}} |\Pi(\mathrm d\theta) \mathrm d x \leq &  \sup_{t\in(a,b)} |h'(t)| \int_a^b  \int_{x-a}^\infty \theta\mathbf 1_{\{\theta\leq 1\}}  \Pi(\mathrm d\theta) \mathrm d x \\
\leq &  \sup_{t\in(a,b)} |h'(t)|   \int_{0}^1 \theta^2 \Pi(\mathrm d\theta).
 \end{split}
\end{equation*}
Third, we have for any $\delta\in(0,b-a)$, 
\begin{equation*}
\begin{split}
\int_{a+\delta}^b & \int_{x-a}^\infty  |h(x-\theta)-h(x) |  \Pi(\mathrm d\theta)\mathrm dx \\
\leq &
\int_{a+\delta}^b |h(x) | \Pi(x-a,\infty) \mathrm dx + \int_{a+\delta}^b \int_{x-a}^\infty |h(x-\theta) | \Pi(\mathrm d\theta)\mathrm dx \\
\leq & \Pi(\delta,\infty)  \int_{a+\delta}^b |h(x) |  \mathrm dx  + \int_{a+\delta}^b \int_{x-a}^{b-a} |h(x-\theta) | \Pi(\mathrm d\theta)\mathrm dx   + \int_{a+\delta}^b \int_{b-a}^\infty |h(x-\theta) | \Pi(\mathrm d\theta)\mathrm dx \\
 < & \infty,
\end{split}
\end{equation*}   
where the last inequality is due to Definition \ref{def_funcspace}(ii), inequality \eqref{levymeasure} and because by using Fubini,
\begin{equation*}
\begin{split}
\int_{a+\delta}^b \int_{x-a}^{b-a} |h(x-\theta) | \Pi(\mathrm d\theta)\mathrm dx  \leq &
\sup_{t\in(2a+\delta-b,a) }|h(t)| \int_{a+\delta}^b \int_{x-a}^{b-a} \Pi(\mathrm d\theta)\mathrm dx \\
= & \sup_{t\in(2a+\delta-b,a) }|h(t)| \int_{\delta}^{b-a}   (\theta-\delta)\Pi(\mathrm d\theta).
\end{split}
\end{equation*}   
Further,
\begin{multline*}
\int_a^{a+\delta}  \int_{\delta}^\infty |h(x-\theta)-h(x) | \Pi(\mathrm d\theta)\mathrm dx \\ \leq  
\Pi(\delta,\infty) \int_a^{a+\delta} |h(x) |  \mathrm dx + \int_a^{a+\delta}  \int_{\delta}^\infty |h(x-\theta) | \Pi(\mathrm d\theta)\mathrm dx
\end{multline*}
and the right hand side is finite by Definition \ref{def_funcspace}(ii).
Lastly, choosing $\delta>0$  small enough such that $h$ has a bounded density in $(a-\delta,b)$, which we denote, with abuse of notation, by $h'$,  we get by Taylor's theorem and Fubini,
\begin{equation*}
\begin{split}
\int_a^{a+\delta} \int_{x-a}^{\delta} |h(x-\theta)-h(x) | \Pi(\mathrm d\theta)\mathrm dx 
\leq &   \int_a^{a+\delta}  \sup_{t\in(x-\delta,x)} |h'(t)| \int_{x-a}^{\delta} \theta  \Pi(\mathrm d\theta)\mathrm dx \\
\leq & \sup_{t\in(a-\delta,b)} |h'(t)| \int_a^{a+\delta}  \int_{x-a}^{\delta} \theta  \Pi(\mathrm  d\theta ) \mathrm dx \\
 = & \sup_{t\in(a-\delta,b)} |h'(t)| \int_0^{\delta}  \theta^2 \Pi(\mathrm  d\theta ).  
\end{split}
\end{equation*}
Combining everything and recalling \eqref{levymeasure}, gives us  $\int_a^b |\mathcal A h(x)| \mathrm dx<\infty$.

\medskip

We now assume that condition (ii) holds, i.e. $\int_0^1 \theta\Pi(\mathrm d\theta)<\infty$. . Then by boundedness of $h'_-$ and $h''$ (if $\sigma>0$) on $(a,b)$, we have  $\int_a^b |(\gamma+ \int_0^1 \theta \Pi(\mathrm d\theta)) h'_-(x) + \frac12 \sigma^2 h''(x) |\mathrm dx<\infty$.
  Further  by Taylor's theorem, we have for all $x\in(a,b)$,
\begin{equation*}
  \int_0^{x-a} |h(x-\theta)-h(x) | \Pi(\mathrm d\theta)    \leq   \sup_{t\in(a,x)} |h'_-(t)|  \int_0^{x-a} \theta\Pi(\mathrm d\theta)   \leq \sup_{t\in(a,b)} |h'_-(t)| \int_0^{b-a} \theta\Pi(\mathrm d\theta).
\end{equation*}
Moreover,
\begin{equation*}
 \begin{split}
   \int_a^b \int_{x-a}^\infty |h(x-\theta)-h(x)| & \Pi(\mathrm d\theta)\mathrm dx  \\  
   \leq & \sup_{t\in(a,b)} |h(t)|\int_a^b \int_{x-a}^\infty  \Pi(\mathrm d\theta)\mathrm dx  + \int_a^b \int_{x-a}^{b-a} |h(x-\theta)| \Pi(\mathrm d\theta)\mathrm dx   \\
      &+ \int_a^b \int_{b-a}^\infty |h(x-\theta)| \Pi(\mathrm d\theta)\mathrm dx \\
= & \sup_{t\in(a,b)} |h(t)|\int_0^\infty \theta\wedge(b-a) \Pi(\mathrm d\theta)   + \sup_{t\in(2a-b,a) }|h(t)| \int_{0}^{b-a}   \theta\Pi(\mathrm d\theta) \\ 
&  + \int_a^b \int_{b-a}^\infty |h(x-\theta)| \Pi(\mathrm d\theta)\mathrm dx
 \end{split}
\end{equation*}
and the right hand side is finite by Definition \ref{def_funcspace}(ii). We conclude that also in this case $\int_a^b |\mathcal A h(x)| \mathrm dx<\infty$.
\end{proof}

\begin{proof}[\textbf{Proof of Theorem \ref{thm_main}}]
Let $x\in[a,b]$.
In order to deal with the possibility that  $\widetilde f$ is not right-continuous at $a$, we introduce the 
function $g:(-\infty,b]\to\mathbb R$ defined by
\begin{equation*}
 g(x) = 
\begin{cases}
\widetilde f(x) & \text{$x\neq a$}, \\
\widetilde f(a+) & \text{$x=a$}. 
\end{cases}
\end{equation*}
For notational convenience, let $T=\tau_a^+\wedge\tau_b^-$.
Then by the regularity assumptions, $g$ is smooth enough on $[a,b]$ in order to use the Meyer-It\^o formula. In particular, using Theorem 70 of \cite{protter_2nded}  in combination with the fact that the continuous part of the quadratic variation of the L\'evy process is given by $[X,X]^c_t=\sigma^2 t$ (cf. \cite{jacodshiryaev}*{Theorem I.4.52, Definition II.2.6 and Corollary II.4.19}) and  (i) Corollary 1 on p.216 of \cite{protter_2nded} in the case where $X$ has paths of unbounded variation or (ii)   Corollary 3 on p.225 of  \cite{protter_2nded} in the case where $X$ has paths of bounded variation,
we get under $\mathbb P_x$,
\begin{equation*}
\begin{split}
\mathrm{e}^{-q(t\wedge T)} g(X_{t\wedge T})  
= & g(X_0)  +  \int_{0+}^{t\wedge T} \mathrm{e}^{-qs} \mathrm d g(X_s) -   \int_{0+}^{t\wedge T} q\mathrm{e}^{-qs} g(X_{s-})  \mathrm{d}s \\
= & g(X_0)  + \int_{0+}^{t\wedge T}\mathrm{e}^{-qs} \left( \frac{\sigma^2}{2} g''(X_{s-})-q g (X_{s-}) \right) \mathrm{d}s    +\int_{0+}^{t\wedge T}\mathrm{e}^{-qs} g'_-(X_{s-}) \mathrm{d}  X_s  \\
 & + \sum_{0<s\leq t\wedge T}\mathrm{e}^{-qs}[\Delta g(X_s)- g'_-(X_{s-})  \Delta X_s  ] \\
= & g(X_0)  + \int_{0+}^{t\wedge T}\mathrm{e}^{-qs}   (\mathcal A-q) g(X_{s-})   \mathrm{d}s \\
&  + \Bigg\{ \int_{0+}^{t\wedge T}\mathrm{e}^{-qs}  g_-'(X_{s-})\mathrm{d} \left( X_s-\gamma s-\sum_{0<u\leq s}\Delta X_u\mathbf{1}_{\{ |\Delta X_u| >1\}} \right)   \Bigg\} \\
&  + \Bigg\{ \sum_{0<s\leq t\wedge T}\mathrm{e}^{-qs} \left( \Delta  g(X_{s-}+\Delta X_s)-  g_-'(X_{s-})\Delta X_s\mathbf{1}_{\{ |\Delta X_s| \leq1\}} \right) \\
& -   \int_{0+}^{t\wedge T} \int_{0+}^{\infty} \mathrm{e}^{-qs} \left(  g (X_{s-}-\theta) - g (X_{s-}) +  g'_-(X_{s-}) \theta \mathbf{1}_{\{0<\theta\leq 1\}} \right) \Pi(\mathrm{d}\theta)\mathrm{d}s \Bigg\}.
\end{split}
\end{equation*}
Here we have used the following notation: $X_{s-}=\lim_{u\uparrow s} X_u$, $\Delta X_s=X_s-X_{s-}$ and $\Delta g(X_s)=g(X_s)-g(X_{s-})$.
By the L\'evy-It\^o decomposition (cf. \cite{kypbook}*{Section 2.1}) the expression between the first pair of curly brackets is a zero-mean martingale and by the compensation formula (cf. \cite{kypbook}*{Corollary 4.6}) the expression between the second pair of curly brackets is also a zero-mean martingale.
Hence taking expectations under $\mathbb P_x$ and letting $t\to\infty$, we have with the aid of the dominated convergence theorem and the regularity properties of $g$ on $[a,b]$, 
\begin{equation}
\label{eq_exitband}
\begin{split}
\mathbb E_x \left[ \mathrm{e}^{-q (\tau_a^-\wedge\tau_b^+) } g(X_{ \tau_a^-\wedge\tau_b^+})  \right] = & g(x)  + \mathbb E_x \left[ \int_{0}^{\tau_a^-\wedge\tau_b^+}\mathrm{e}^{-qs}   (\mathcal A-q) g(X_{s})   \mathrm{d}s \right] \\
= & g(x)  + \int_a^b (\mathcal A-q) g(z)  \int_{0}^{\infty}\mathrm{e}^{-qs}  \mathbb P_x(X_s\in\mathrm dz,s<\tau_a^-\wedge\tau_b^+)  \mathrm{d}s.
\end{split}
\end{equation}
Since by the lack of upward jumps,
\begin{equation}\label{eq_splitting}
\begin{split}
\mathbb E_x \left[ \mathrm{e}^{-q (\tau_a^-\wedge\tau_b^+) } g(X_{ \tau_a^-\wedge\tau_b^+})  \right] = & \mathbb E_x \left[ \mathrm{e}^{-q  \tau_a^- } g(X_{ \tau_a^-}) \mathbf 1_{\{\tau_a^-<\tau_b^+\}} \right]  + g(b) \mathbb E_x \left[ \mathrm{e}^{-q \tau_b^+ } \mathbf 1_{\{\tau_a^->\tau_b^+\}} \right],
\end{split}
\end{equation}
we have by  \eqref{eq_exitband}, \eqref{eq_splitting} and the definition of $g$,
\begin{equation}\label{mainresult_nonspec}
\begin{split}
 \mathbb E_x \left[ \mathrm{e}^{-q  \tau_a^- } f (X_{ \tau_a^-}) \mathbf 1_{\{\tau_a^-<\tau_b^+\}} \right] 
 = & g(x) -  \mathbb E_x \left[ \mathrm{e}^{-q \tau_b^+ } \mathbf 1_{\{\tau_a^->\tau_b^+\}} \right] \widetilde f(b)  \\
 & + \int_a^b (\mathcal A-q) \widetilde f(z)  \int_{0}^{\infty}\mathrm{e}^{-qs}  \mathbb P_x(X_s\in\mathrm dz,s<\tau_a^-\wedge\tau_b^+)  \mathrm{d}s \\ 
& + (f(a)-\widetilde f(a+))\mathbb E_x \left[ \mathrm{e}^{-q  \tau_a^- } \mathbf 1_{\{X_{\tau_a^-}=a,\tau_a^-<\tau_b^+\}} \right].
\end{split}
\end{equation}
Now the identities of the theorem follow by plugging \eqref{twosidedexit}, \eqref{resolvent} and \eqref{creeping} into the above equation, while noting that if $x=X_0=a$ and $X$ has paths of bounded variation, then $X_{\tau_a^-}=a$ is an event which has probability $0$, whereas if $x=X_0=a$ and $X$ has paths of unbounded variation, then $\tau_a^-=0$ and $X_{\tau_a^-}=a$ almost surely, which implies in addition  that $W^{(q)}(0)=0$, cf. \eqref{twosidedexit}.
\end{proof}

\begin{proof}[\textbf{Proof of Corollary \ref{corol_main}}]
The corollary follows easily, since by the extra condition assumed, the last term of \eqref{mainidentity_general} vanishes and by Lemma \ref{lemma_main} we are allowed to split the integral into two terms.
\end{proof}

\section{Overshoot identities for reflected and refracted L\'evy processes}\label{sec_reflecrefrac}

Following the  proof  of Theorem \ref{thm_main}, one can easily establish     identities involving the overshoot of   reflected or refracted spectrally negative L\'evy processes as well. To this end, let $Z=\{Z_t:t\geq0\}$ be the process $X$ reflected at level $b\in\mathbb R$, i.e.
\begin{equation*}
Z_t=X_t- \xi_t, \quad \text{where $\xi_t=\left( \sup_{0\leq s\leq t} (X_s-b)\vee0 \right)$} 
\end{equation*}  
 and define the stopping time
 \begin{equation*}
 T_a^-=\inf\{t>0:Z_t<a\}.
 \end{equation*}
 Further, let  $U=\{U_t:t\geq0\}$ be the process $X$ refracted at level $c\in\mathbb R$, i.e. $U$ is the strong solution to the stochastic differential equation,
 \begin{equation}\label{def_refr}
 \mathrm d U_t= \mathrm d X_t-  \delta \mathbf 1_{\{U_t>c\}} \mathrm d t, \quad U_0=X_0,
 \end{equation}  
where $0<\delta<\gamma+\int_0^1\theta\Pi(\mathrm d\theta)$. By Theorem 1 of \cite{kyploeffen}, the process $U$ is well-defined. We denote the first passage times of $U$ above and below a level by 
 \begin{equation*}
 \kappa_a^+ = \inf\{t>0:U_t>a\}, \quad \text{and} \quad \kappa_a^- = \inf\{t>0:U_t<a\}.
 \end{equation*}
 Then under the   conditions of Theorem \ref{thm_main}, we have for $a<x\leq b$,
 \begin{equation}\label{overshoot_refl}
 \begin{split}
   \mathbb E_x  \Big[ \mathrm e^{-q T_a^-}  f(Z_{T_a^-})  \Big]  
= & \widetilde f(x)  - \mathbb E_x \left[ \int_0^{T_a^-} \mathrm e^{-q s} \mathrm d \xi_s \right]    \widetilde f'_-(b) \\
& + \int_a^b   (\mathcal A-q) \widetilde f(z) \left[ \int_{0}^{\infty}\mathrm{e}^{-qs}  \mathbb P_x(Z_s\in\mathrm dz,s<T_a^-)  \mathrm{d}s \right] \mathrm dz \\
 & +   \left(   f(a) - \widetilde f(a+) \right) \mathbb E_x \left[ \mathrm e^{-q T_a^-}   \mathbf 1_{\{X_{T_a^-}=a \}}  \right] 
 \end{split}
 \end{equation}
and for $c\in(a,b)$,
\begin{equation}\label{overshoot_refr}
 \begin{split}
   \mathbb E_x  \Big[ \mathrm e^{-q \kappa_a^-}  f&(U_{\kappa_a^-})  \mathbf 1_{\{ \kappa_a^-<\kappa_b^+ \}}  \Big]   \\
 = & \widetilde f(x)  - \mathbb E_x \left[ \mathrm{e}^{-q \kappa_b^+ } \mathbf 1_{\{\kappa_b^+<\kappa_a^-\}} \right]    \widetilde f(b) \\
& + \int_a^b   \left( (\mathcal A-q) \widetilde f(z) - \delta \mathbf 1_{\{z>c\}} \widetilde f'_-(z) \right) \left[ \int_{0}^{\infty}\mathrm{e}^{-qs}  \mathbb P_x(U_s\in\mathrm dz,s<\kappa_a^-\wedge\kappa_b^+)  \mathrm{d}s \right] \mathrm dz \\
 & +  \left(  f(a) - \widetilde f(a+) \right) \mathbb E_x \left[ \mathrm e^{-q \kappa_a^-}   \mathbf 1_{\{U_{\kappa_a^-}=a, \kappa_a^-<\kappa_b^+ \}}  \right].
 \end{split}
\end{equation}
The proof of these two identities is almost the same as the proof of \eqref{mainresult_nonspec} given in Section \ref{sec_proof} and we leave the details to the reader.
Note that all the expectations and resolvent measures on the right hand sides of \eqref{overshoot_refl} and \eqref{overshoot_refr} admit analytic expressions in terms of scale functions. In particular, see Theorem 10.3 in \cite{kypbook}, respectively   Theorem 4(i) in \cite{kyploeffen}, for the first expectation on the right hand side of \eqref{overshoot_refl}, respectively \eqref{overshoot_refr}. Further, see Theorem 1.(ii) in \cite{pistorius_exitergod} and Theorem 6(i) in \cite{kyploeffen} for the two  $q$-resolvent measures and note that the  two expectations involving the event of creeping are non-zero if and only if the Gaussian coefficient is non-zero and in that case, these expectations can be derived from the corresponding resolvent measure,  see the proof of Corollary 2 in   \cite{pistorius_potential}. 
Note  also that   \eqref{overshoot_refl} for the special case  $\widetilde f(y)=f'_-(a)y + f(a)$, $y\in(a,b]$ is given in Proposition 5.5 of \cite{palmgerbershiu}.

\section{Examples}\label{sec_examples}

In order to illustrate why \eqref{mainidentity_general} and \eqref{mainidentity_simple} are useful with a simple example, let us take $f(y)=1$ for  $y\leq a$  in \eqref{gerbershiu}. It is well known, see e.g \cite{kypbook}*{Theorem 8.1} or combine \eqref{twosidedexit} and \eqref{resolvent}, that for $x\in[a,b]$,
\begin{equation}\label{simpleexample}
 \mathbb E_x \Big[ \mathrm e^{-q \tau_a^-}  \mathbf 1_{\{ \tau_a^-<\tau_b^+ \}}  \Big] 
 = Z^{(q)}(x-a) - \frac{ W^{(q)}(x-a)}{W^{(q)}(b-a)}  Z^{(q)}(b-a),
\end{equation}
where $Z^{(q)}(y)=1+q\int_0^y W^{(q)}(z)\mathrm d z$, $y\in\mathbb R$. Using \eqref{extis0}, which corresponds to choosing the extension $\widetilde f(y)=0$ for $a<y\leq b$,   we get for $a<x\leq b$,
\begin{equation}\label{example_extis0}
 \begin{split}
   \mathbb E_x  \Big[ \mathrm e^{-q \tau_a^-}   & \mathbf 1_{\{ \tau_a^-<\tau_b^+ \}}  \Big]  
 \\
 = &   \int_a^b      \Pi(z-a,\infty)   \left[ \frac{W^{(q)}(x-a)}{W^{(q)}(b-a)} W^{(q)}(b- z)  -  W^{(q)}(x-z) \right] \mathrm dz \\
 & +    \frac{\sigma^2}2 \left( W^{(q)\prime}(x-a) - \frac{W^{(q)}(x-a)}{W^{(q)}(b-a)} W^{(q)\prime}(b-a)  \right). 
 \end{split}
\end{equation}
If we did not know the identity \eqref{simpleexample}, it would  not at all be obvious that   the right hand side of \eqref{example_extis0}  actually simplifies  to the right hand side of \eqref{simpleexample}. If we instead use the extension $\widetilde f(y)=1$ for $a<y\leq b$, then Corollary \ref{corol_main} gives us directly \eqref{simpleexample}, since for that choice $\mathcal (A-q)\widetilde f(z)=-q$, $a< z< b$.

\bigskip

We now consider a more interesting example.
Fix $\delta>0$ and let $Y=\{Y_t:t\geq0\}$ be the spectrally negative L\'evy process defined by $Y_t=X_t-\delta t$ and denote by $\mathbb W^{(q)}(x)$ the $q$-scale function of $Y$. With
 \begin{equation*}
 \nu_a^+ = \inf\{t>0:Y_t>a\}, \quad \text{and} \quad \nu_a^- = \inf\{t>0:Y_t<a\},
 \end{equation*}
we look at the following special case of \eqref{gerbershiu},
\begin{equation} \label{overshootscaledef}
 \mathbb E_x \Big[ \mathrm e^{-p \nu_a^-}  W^{(q)}(Y_{\nu_a^-}) \mathbf 1_{\{ \nu_a^-<\nu_b^+ \}}  \Big],
\end{equation}
where $p,q\geq0$, $0\leq a\leq x< b$ and recall $W^{(q)}(x)$ is the scale function of $X$. This particular case has appeared in the study of refracted L\'evy processes and  occupation times of (refracted) spectrally negative L\'evy processes, see \citelist{\cite{kyploeffen}\cite{landriault_occup}\cite{loeffenrenaudzhou}\cite{kyp_occup}\cite{renaud_red}}.  We want to obtain an as simple as possible expression for \eqref{overshootscaledef}. For this we use  \eqref{mainidentity_simple} with  $\widetilde f(y)=W^{(q)}(y)$ being the obvious choice for the extension. 
 We first restrict ourselves to the case were  $X$ (equivalently $Y$) has paths of bounded variation or $\sigma>0$.
In that case $W^{(q)}(x)$ (restricted to $(-\infty,b]$) lies in $\mathcal H(X;a,b)=\mathcal H(Y;a,b)$ for any $0<a<b$, cf. Equations (4)-(6) and Theorem 1 in \cite{chankypsavov}. Further, when $X$ has paths of bounded variation, one can easily show by taking Laplace transforms and using \eqref{def_scale} and \cite{kypbook}*{Lemma 8.6}  that $(\mathcal A-q)W^{(q)}(x)=0$ for almost every $x>0$. If $\sigma>0$, then  $(\mathcal A-q)W^{(q)}(x)=0$ for every $x>0$, see e.g. p.694 of \cite{chankypsavov}. Combining everything and denoting $\mathbb A h(x) = \mathcal Ah(x)-\delta h'_-(x)$ for $h\in\mathcal H(X;a,b)$, we   conclude that for $0<a\leq x<b$, 
\begin{equation}\label{overshootscale}
 \begin{split}
 \mathbb E_x  \Big[ \mathrm e^{-p \nu_a^-}   W^{(q)}&(Y_{\nu_a^-}) \mathbf 1_{\{ \nu_a^-<\nu_b^+ \}}  \Big] \\
= & W^{(q)}(x)  -   \int_a^x   (\mathbb A-p) W^{(q)}(z)  \mathbb  W^{(p)}(x- z) \mathrm dz \\
& -  \frac{\mathbb W^{(p)}(x-a)}{\mathbb W^{(p)}(b-a)} \left[ W^{(q)}(b)  - \int_a^b   (\mathbb A-p) W^{(q)}(z)   \mathbb  W^{(p)}(b- z)  \mathrm dz \right] \\
= & W^{(q)}(x)  -   \int_a^x \left(  (q-p) W^{(q)}(z)-\delta W^{(q)\prime}(z) \right) \mathbb  W^{(p)}(x- z) \mathrm dz \\
& -  \frac{\mathbb W^{(p)}(x-a)}{\mathbb W^{(p)}(b-a)} \left[ W^{(q)}(b)  - \int_a^b   \left(  (q-p) W^{(q)}(z)-\delta W^{(q)\prime}(z) \right)    \mathbb  W^{(p)}(b- z)  \mathrm dz \right].
 \end{split}
\end{equation}
By taking limits as $a\downarrow0$, we see that  \eqref{overshootscale} also holds when $a=0$.
Note that in the above expression  the derivative of the scale function appears, despite
the fact that in general $W^{(q)\prime}(x)$ may not be well defined for a countable number of points. However, since it only appears in the integrand of an ordinary Lebesgue integral, this does not present a problem.
In the other case where $X$ has paths of unbounded variation and $\sigma=0$, it is unknown if $W^{(q)}\in \mathcal H(X;a,b)$, but one can show that  \eqref{overshootscale} (ignoring the middle part)  also follows in this case from \eqref{mainidentity_simple} when used in combination with a mollification argument, see the appendix. 

Equation \eqref{overshootscale} has been derived earlier, see \cite{kyploeffen}*{Theorem 16} in combination with \eqref{extis0}  for  the case $p=q$ and $X$ having paths of bounded variation, see \cite{loeffenrenaudzhou}*{Section 2} for the case $\delta=0$ and see \cite{renaud_red}*{Lemma 1} in combination with \cite{loeffenrenaudzhou}*{Equation (6)}  for the general case. Whereas in these references it took quite some effort to get to  \eqref{overshootscale}, with Theorem \ref{thm_main} or Corollary \ref{corol_main} it is  obvious how to come up with this relatively simple expression.
We remark that getting as simple as possible, analytic expressions for \eqref{overshootscaledef} (and in general \eqref{gerbershiu}) is not just useful for evaluating this expectation. It also allows one to tackle more complicated cases, compare e.g. the main results of \cite{landriault_occup} and \cite{loeffenrenaudzhou}. Further in the context of refracted L\'evy processes, obtaining \eqref{overshootscale} (for the case $p=q$) was a crucial step in \cite{kyploeffen} for ultimately showing the existence of these processes in the case where $X$ has paths of unbounded variation and also for solving the related optimal control problem with bounded dividend rates, cf. \cite{kyploeffenperez}.

\section{Appendix}
Here we show that \eqref{overshootscale} also holds when $X$ has paths of unbounded variation with no Gaussian component by using  \eqref{mainidentity_simple} and   mollification.  Note that in this case $W^{(q)}(x)$ is continuously differentiable on $(0,\infty)$ and $\lim_{x\downarrow0}  W^{(q)\prime} (x)=\infty$.

There exists a function $\phi$ which is infinitely differentiable function, has support $[-2,-1]$ and  satisfies $\int_{-\infty}^{\infty}\phi(x)=1$, see e.g. Section I.1.2 in \cite{hörmander}.
For $\epsilon>0$, define $\phi_\epsilon(x)=\frac{1}{\epsilon}\phi(x/\epsilon)$. Then $\phi_\epsilon$ is infinitely differentiable  with support $[-2\epsilon,-\epsilon]$ and $\int_{-\infty}^{\infty}\phi_\epsilon(x)=1$. Consider
\begin{equation*}
(W^{(q)}\star \phi_\epsilon)(x):=\int_{-\infty}^{\infty} W^{(q)}(x-z)\phi_\epsilon(z)\mathrm d z.
\end{equation*}
The function $(W^{(q)}\star \phi_\epsilon)$  is easily shown to be infinitely differentiable on $\mathbb R$, see e.g. \cite{hörmander}*{Theorem I.1.3.1}. 
Since $W^{(q)}(x)$ is continuously differentiable on $(0,\infty)$ and the derivative $W^{(q)\prime}$  is bounded on sets of the form $[1/n, n]$, $n>0$, we get by a standard application of the mean value theorem and the dominated convergence theorem that for each $x>0$, there exists $\bar{\epsilon}$ such that for all $0<\epsilon<\bar{\epsilon}$,
\begin{equation*}
(W^{(q)}\star \phi_\epsilon)'(x) = (W^{(q)\prime}\star \phi_\epsilon)(x).
\end{equation*}  
Therefore,
\begin{equation*}
\begin{split}
|W^{(q)\prime}(x) - (W^{(q)}\star \phi_\epsilon)'(x)  | = & \left| \int_{-\infty}^{\infty} \left(W^{(q)\prime}(x) - W^{(q)\prime}(x-z) \right) \phi_\epsilon(z) \mathrm d z  \right| \\
\leq & \sup_{t\in[x+\epsilon,x+2\epsilon]}  \left| W^{(q)\prime}(x) - W^{(q)\prime}(t) \right|
\end{split}
\end{equation*}
and thus  $\lim_{\epsilon\downarrow0}(W^{(q)}\star \phi_\epsilon)'(x)= W^{(q)\prime}(x)$ for all $x>0$. 
Further, as
\begin{equation*}
(W^{(q)}\star \phi_\epsilon)(x)=  \int_{-2}^{-1} W^{(q)}(x-\epsilon z)\phi_1(z)\mathrm d z 
\end{equation*}
and $W^{(q)}$ is an increasing function, we have that $(W^{(q)}\star \phi_\epsilon)(x)$ decreases to   $W^{(q)}(x)$ as $\epsilon\downarrow 0$ for all $x\in\mathbb R$. 

Next, we prove that $(\mathcal A -q)(W^{(q)}\star \phi_\epsilon)(x)=0$ for all $x>0$.
We know that for each $x,a,b\in\mathbb R$ with $a\leq b$, the process
\begin{equation*}
t\mapsto \mathrm e^{-q(t\wedge\tau_a^-\wedge\tau_b^+)} W^{(q)} \left( X_{t\wedge\tau_a^-\wedge\tau_b^+}  - a \right)
\end{equation*}
is a $\mathbb P_x$-martingale, see e.g Remark 5 in \cite{avramkyppist}. Then for $c\in[a,b]$, $\tau_c^-\leq \tau_a^-$ and so
\begin{equation*}
t\mapsto \mathrm e^{-q(t\wedge\tau_c^-\wedge\tau_b^+)} W^{(q)} \left( X_{t\wedge\tau_c^-\wedge\tau_b^+}  - a \right)
\end{equation*}
is a $\mathbb P_x$-martingale as well, cf. \cite{dellacheriemeyerB}*{Theorem VI.12}.
Then by Tonelli  we have for every $F\in\mathcal F_s$ and $t>s$,
\begin{equation*}
\begin{split}
\mathbb E_x \Big[  \mathrm e^{-q(t\wedge\tau_0^-\wedge\tau_b^+)} (W^{(q)}\star \phi_\epsilon) & \left( X_{t\wedge\tau_0^-\wedge\tau_b^+} \right) \mathbf 1_{F} \Big] \\
= & \int_{-2\epsilon}^{-\epsilon}\mathbb E_x \left[ \mathrm e^{-q(t\wedge\tau_0^-\wedge\tau_b^+)} W^{(q)} (X_{t\wedge\tau_0^-\wedge\tau_b^+} - z)    \mathbf 1_{F} \right] \phi_\epsilon(z)\mathrm d z \\
= & \int_{-2\epsilon}^{-\epsilon}\mathbb E_x \left[ \mathrm e^{-q(s\wedge\tau_0^-\wedge\tau_b^+)} W^{(q)} (X_{s\wedge\tau_0^-\wedge\tau_b^+} - z)   \mathbf 1_{F}  \right] \phi_\epsilon(z)\mathrm d z \\
= & \mathbb E_x \left[ \mathrm e^{-q(s\wedge\tau_0^-\wedge\tau_b^+)} (W^{(q)}\star \phi_\epsilon) \left( X_{s\wedge\tau_0^-\wedge\tau_b^+} \right)  \mathbf 1_{F}  \right]. 
\end{split}
\end{equation*} 
It follows that for any $x\in\mathbb R$ and $b>0$,
\begin{equation*}
t\mapsto  \mathrm e^{-q(t\wedge\tau_0^-\wedge\tau_b^+)} (W^{(q)}\star \phi_\epsilon)  \left( X_{t\wedge\tau_0^-\wedge\tau_b^+} \right)
\end{equation*}
is a $\mathbb P_x$-martingale as well. Now since $(W^{(q)}\star \phi_\epsilon)$ is smooth, we get $(\mathcal A -q)(W^{(q)}\star \phi_\epsilon)(x)=0$ for all $x>0$, cf.  p.694 of \cite{chankypsavov}.
Then using \eqref{mainidentity_simple} with  $\widetilde f(y)=(W^{(q)}\star \phi_\epsilon)(y)$, we get for $0<a\leq x<b$ and $p,q\geq0$,
\begin{equation*} 
 \begin{split}
 \mathbb E_x  \Big[ \mathrm e^{-p \nu_a^-}   & (W^{(q)}\star \phi_\epsilon)  (Y_{\nu_a^-}) \mathbf 1_{\{ \nu_a^-<\nu_b^+ \}}  \Big] \\
= & (W^{(q)}\star \phi_\epsilon)(x)  -   \int_a^x \left(  (q-p) (W^{(q)}\star \phi_\epsilon)(z)-\delta (W^{(q)}\star \phi_\epsilon)'(z) \right) \mathbb  W^{(p)}(x- z) \mathrm dz \\
& -  \frac{\mathbb W^{(p)}(x-a)}{\mathbb W^{(p)}(b-a)} \Big[ (W^{(q)}\star \phi_\epsilon)(b)  \\
& - \int_a^b   \left(  (q-p) (W^{(q)}\star \phi_\epsilon)(z)-\delta (W^{(q)}\star \phi_\epsilon)'(z) \right)    \mathbb  W^{(p)}(b- z)  \mathrm dz \Big].
 \end{split}
\end{equation*}
Taking the limit as $\epsilon\downarrow0$, we conclude by the dominated convergence theorem that \eqref{overshootscale} also holds when $X$ has paths of unbounded variation with no Gaussian component.


\begin{bibdiv}
\begin{biblist}

\bib{avramkyppist}{article}{
   author={Avram, F.},
   author={Kyprianou, A. E.},
   author={Pistorius, M. R.},
   title={Exit problems for spectrally negative L\'evy processes and
   applications to (Canadized) Russian options},
   journal={Ann. Appl. Probab.},
   volume={14},
   date={2004},
   number={1},
   pages={215--238},
   %
   %
   %
}

\bib{palmgerbershiu}{misc}{
author={Avram, F.},
author={Palmowski, Z.},
author={Pistorius, M.R.},
title={On Gerber-Shiu functions and optimal dividend distribution for a L\'evy risk-process in the presence of a penalty function - a probabilistic approach},
note={To appear in Ann. Appl. Probab., arXiv:1110.4965v4 [math.PR]},
date={2014},
}

%
%
%
%
%
%
%
%
%
%
%
%

%
%
%
%
%
%
%
%
%
%
%
%
%
%


%
%
%
%
%
%
%
%
%

\bib{chankypsavov}{article}{
   author={Chan, T.},
   author={Kyprianou, A. E.},
   author={Savov, M.},
   title={Smoothness of scale functions for spectrally negative L\'evy
   processes},
   journal={Probab. Theory Related Fields},
   volume={150},
   date={2011},
   number={3-4},
   pages={691--708},
   %
   %
   %
}

\bib{dellacheriemeyerB}{book}{
   author={Dellacherie, Claude},
   author={Meyer, Paul-Andr{\'e}},
   title={Probabilities and potential. B},
   series={North-Holland Mathematics Studies},
   volume={72},
   note={Theory of martingales;
   Translated from the French by J. P. Wilson},
   publisher={North-Holland Publishing Co., Amsterdam},
   date={1982},
   pages={xvii+463},
   %
   %
}

\bib{gerbershiu}{article}{
   author={Gerber, Hans U.},
   author={Shiu, Elias S. W.},
   title={On the time value of ruin},
   note={With discussion and a reply by the authors},
   journal={N. Am. Actuar. J.},
   volume={2},
   date={1998},
   number={1},
   pages={48--78},
   %
   %
   %
}


\bib{hörmander}{book}{
   author={H{\"o}rmander, Lars},
   title={The analysis of linear partial differential operators. I},
   series={Grundlehren der Mathematischen Wissenschaften [Fundamental
   Principles of Mathematical Sciences]},
   volume={256},
   note={Distribution theory and Fourier analysis},
   publisher={Springer-Verlag, Berlin},
   date={1983},
   pages={ix+391},
   %
   %
   %
}
	
\bib{jacodshiryaev}{book}{
	   author={Jacod, Jean},
	   author={Shiryaev, Albert N.},
	   title={Limit theorems for stochastic processes},
	   series={Grundlehren der Mathematischen Wissenschaften [Fundamental
	   Principles of Mathematical Sciences]},
	   volume={288},
	   edition={2},
	   publisher={Springer-Verlag, Berlin},
	   date={2003},
	   %
	   %
	   %
	   %
}
	
	
	\bib{kuznetsovkyprivero}{article}{
	author={Kuznetsov, A.},
	author={Kyprianou, A.E.},
	author={Rivero, V.},
	title={The theory of scale functions for spectrally negative L\'evy processes},
	%
	book={
	title={L\'evy matters II},
	series={Lecture Notes in Mathematics},
	volume={2061},
	publisher={Springer, Heidelberg},
%
%
%
%
%
	%
	},
	date={2012},
	pages={97--186},
	}
	
\bib{kypbook}{book}{
   author={Kyprianou, Andreas E.},
   title={Fluctuations of L\'evy processes with applications},
   series={Universitext},
   edition={2},
   note={Introductory lectures},
   publisher={Springer, Heidelberg},
   date={2014},
   %
   %
   %
   %
   %
}

\bib{kyploeffen}{article}{
   author={Kyprianou, A. E.},
   author={Loeffen, R. L.},
   title={Refracted L\'evy processes},
   language={English, with English and French summaries},
   journal={Ann. Inst. Henri Poincar\'e Probab. Stat.},
   volume={46},
   date={2010},
   number={1},
   pages={24--44},
   %
   %
   %
}

\bib{kyploeffenperez}{article}{
   author={Kyprianou, Andreas E.},
   author={Loeffen, Ronnie},
   author={P{\'e}rez, Jos{\'e}-Luis},
   title={Optimal control with absolutely continuous strategies for
   spectrally negative L\'evy processes},
   journal={J. Appl. Probab.},
   volume={49},
   date={2012},
   number={1},
   pages={150--166},
   %
   %
   %
}

\bib{kyp_occup}{misc}{
author={Kyprianou, A.E.},
author={Pardo, J.C.},
author={P\'erez, J.L.},
title={Occupation times of refracted L\'evy processes},
%
note={To appear in J. Theor. Probab.},
date={2013},
}

\bib{landriault_occup}{article}{
   author={Landriault, D.},
   author={Renaud, J.-F.},
   author={Zhou, X.},
   title={Occupation times of spectrally negative L\'evy processes with
   applications},
   journal={Stochastic Process. Appl.},
   volume={121},
   date={2011},
   number={11},
   pages={2629--2641},
   %
   %
   %
}

%
%
%
%
%
%
%
%
%
%
%
%
%

\bib{loeffenrenaudzhou}{article}{
   author={Loeffen, Ronnie L.},
   author={Renaud, Jean-Fran{\c{c}}ois},
   author={Zhou, Xiaowen},
   title={Occupation times of intervals until first passage times for
   spectrally negative L\'evy processes},
   journal={Stochastic Process. Appl.},
   volume={124},
   date={2014},
   number={3},
   pages={1408--1435},
   %
   %
   %
}

\bib{protter_2nded}{book}{
   author={Protter, Philip E.},
   title={Stochastic integration and differential equations},
   series={Applications of Mathematics (New York)},
   volume={21},
   edition={2},
   note={Stochastic Modelling and Applied Probability},
   publisher={Springer-Verlag, Berlin},
   date={2004},
   %
   %
   %
}

\bib{pistorius_exitergod}{article}{
   author={Pistorius, M. R.},
   title={On exit and ergodicity of the spectrally one-sided L\'evy process
   reflected at its infimum},
   journal={J. Theoret. Probab.},
   volume={17},
   date={2004},
   number={1},
   pages={183--220},
   %
   %
   %
}

\bib{pistorius_potential}{article}{
   author={Pistorius, Martijn R.},
   title={A potential-theoretical review of some exit problems of spectrally
   negative L\'evy processes},
   conference={
      title={S\'eminaire de Probabilit\'es XXXVIII},
   },
   book={
      series={Lecture Notes in Math.},
      volume={1857},
      publisher={Springer, Berlin},
   },
   date={2005},
   pages={30--41},
   %
}

%
%
%
%
%
%
%
%
%
%
%
%
%
%


 
\bib{renaud_red}{article}{
   author={Renaud, Jean-Fran{\c{c}}ois},
   title={On the time spent in the red by a refracted L\'evy risk process},
   journal={J. Appl. Probab.},
   volume={51},
   date={2014},
   number={4},
   pages={1171--1188},
   %
   %
   %
}

\bib{sato}{book}{
   author={Sato, Ken-iti},
   title={L\'evy processes and infinitely divisible distributions},
   series={Cambridge Studies in Advanced Mathematics},
   volume={68},
   note={Translated from the 1990 Japanese original;
   Revised by the author},
   publisher={Cambridge University Press, Cambridge},
   date={1999},
   %
   %
   %
}

\end{biblist}
\end{bibdiv}
\end{document}